\documentclass[
english
]{amsart}

\usepackage{amssymb}
\usepackage{amsthm}
\usepackage{amsfonts}
\usepackage{babel}
\usepackage{amsmath}
\usepackage{longtable} 
\usepackage{array} 
\usepackage{setspace}
\usepackage{datetime}
\usepackage{ifpdf}
\usepackage{textcomp}
\usepackage{euscript}
\usepackage{url}
\usepackage{enumerate}

\usepackage{pslatex}          
\usepackage[pdftex]{graphicx}
\usepackage[pdftex,breaklinks]{hyperref} 

\newtheorem{theorem}{Theorem}[section]
\newtheorem{lemma}[theorem]{Lemma}

\theoremstyle{definition}
\newtheorem{definition}[theorem]{Definition}

\numberwithin{equation}{section}

\def \E{\mbox{$ I\!\! E$}}
\def \M{\mbox{$ I\!\! M$}}

\def \ksW{K.~S.~Williams}
\def \nK{N.~Koblitz}

\hypersetup{
pdfinfo={
   Author={Ay\c{s}e Alaca, \c{S}aban Alaca and Eb\'{e}n\'{e}zer Ntienjem},
   Title={Evaluation of the Convolution Sum involving the Sum of 
Divisors Function for 14, 22 and 26},
   CreationDate={D:20151003092000},
   ModDate={D:\pdfdate)=},
   Subject={2010 Mathematics Subject Classification: 11A25, 11E20, 11E25, 11F11, 11F20, 11F27},
   Keywords={Sums of Divisors; Dedekind eta function; Convolution Sums; 
Modular Forms; Eisenstein forms; Cusp Forms; Octonary quadratic Forms; 
Number of Representations}
}
}

\author{Ay\c{s}e Alaca, \c{S}aban Alaca and Eb\'{e}n\'{e}zer Ntienjem}
\address{
Centre for Research in Algebra and Number Theory \\
School of Mathematics and Statistics\\
Carleton University\\
1125 Colonel By Drive\\
Ottawa, Ontario, K1S 5B6, Canada
}
\email{aalaca@math.carleton.ca}

\email{salaca@math.carleton.ca}

\email{ebenezer.ntienjem@carleton.ca;ntienjem@gmail.com}


\begin{document}

\title[Evaluation of the Convolution Sum for 14, 22 and 26]
{Evaluation of the Convolution Sum involving the Sum of Divisors Function 
for 14, 22 and 26
}

\begin{abstract}
For all natural numbers $n$, we discuss the evaluation of the convolution sum, 
$\underset{\substack{
 {(l,m)\in\mathbb{N}_{0}^{2}} \\ {\alpha\,l+\beta\,m=n}
} }{\sum}\sigma(l)\sigma(m)$, 
where $\alpha\beta=14,22,26$. 
We generalize the extraction of the convolution sum using Eisenstein 
forms of weight $4$ for all pairs of positive integers $(\alpha,\beta)$.
We also determine formulae for the number of representations of a 
positive integer by the octonary quadratic forms
$a\,(x_{1}^{2} + x_{2}^{2} + x_{3}^{2} + x_{4}^{2})+ b\,(x_{5}^{2} + x_{6}^{2}
  + x_{7}^{2} + x_{8}^{2})$, 
where $(a,b)= (1,1),(1,3),(1,9),(2,3)$. These numbers of representations of 
a positive integer are applications of the evaluation of certain convolution 
sums by J.~G.~Huard et al.\ \cite{huardetal}, A.~Alaca et al.\  
\cite{alaca_alaca_williams2006,alaca_alaca_williams2007a} and D.~Ye 
\cite{ye2015}.
\end{abstract}

\maketitle

\section{Introduction} \label{introduction}

Let $\mathbb{N}$, $\mathbb{N}_{0}$, $\mathbb{Z}$, $\mathbb{Q}$, $\mathbb{R}$ 
and $\mathbb{C}$ denote the sets of
positive integers, non-negative integers, integers, rational numbers, real 
numbers and complex numbers, respectively. 

Let $k, n\in\mathbb{N}$. The sum of positive divisors of $n$ to the power 
of $k$, $\sigma_{k}(n)$, is defined by 
\begin{equation} \label{def-sigma_k-n}
 \sigma_{k}(n) = \sum_{0<d|n}d^{k}.
 \end{equation} 
We  write $\sigma(n)$ as a synonym for $\sigma_{1}(n)$. 
For $m\notin\mathbb{N}$ we set $\sigma_{k}(m)=0$. 

Let $\alpha,\beta\in\mathbb{N}$ be such that $\alpha\leq\beta$.
The convolution sum, 
 $W_{(\alpha,\beta)}(n)$, is defined by 
 \begin{equation} \label{def-convolution_sum}
 W_{(\alpha, \beta)}(n) =  \sum_{\substack{
 {(l,m)\in\mathbb{N}_{0}^{2}} \\ {\alpha\,l+\beta\,m=n}
} }\sigma(l)\sigma(m).
 \end{equation}
We write $W_{\beta}(n)$ as a synonym for $W_{(1,\beta)}(n)$.

The values of $(\alpha,\beta)$ for those convolution sums 
$W_{(\alpha, \beta)}(n)$ that have so far been evaluated are given in 
\autoref{introduction-table-1}. 
We discuss the evaluation of the convolution sum for $\alpha\beta=14,22,26$,
i.e., $(\alpha,\beta)=(2,7),(1,22),(2,11),(1,26),(2,13)$. These convolution 
sums have not been evaluated yet as one can notice from 
\autoref{introduction-table-1}. We also discuss the generalization of the 
extraction of the convolution sum using Eisenstein forms of weight $4$ 
for all pairs of positive integers $(\alpha,\beta)$. 

As an application, convolution sums are used to
determine explicit formulae for the number of representations of a positive 
integer $n$ by the octonary quadratic forms 
\begin{equation} \label{introduction-eq-1}
a\,(x_{1}^{2} +x_{2}^{2} + x_{3}^{2} + x_{4}^{2})+ b\,(x_{5}^{2} + x_{6}^{2} + 
x_{7}^{2} + x_{8}^{2}),
\end{equation}
and 
\begin{equation} \label{introduction-eq-2}
c\,(x_{1}^{2} + x_{1}x_{2} + x_{2}^{2} + x_{3}^{2} + x_{3}x_{4} + x_{4}^{2})
+ d\,(x_{5}^{2} + x_{5}x_{6}+ x_{6}^{2} + x_{7}^{2} + x_{7}x_{8}+ 
x_{8}^{2}),
\end{equation}
respectively, where $a,b,c,d\in \mathbb{N}$. 

So far known explicit formulae for the number of representations 
of $n$ by the octonary form \autoref{introduction-eq-1} are 
referenced in \autoref{introduction-table-rep2}. 
We determine formulae for the number of representations 
of a positive integer $n$ by the octonary quadratic form
\autoref{introduction-eq-1} for which $(a,b)=(1,1),
(1,3),(2,3),(1,9)$. These new results 
are applications of the evaluation of some convolution sums by 
J.~G.~Huard et al.\ \cite{huardetal}, A.~Alaca et al.\    
\cite{alaca_alaca_williams2006,alaca_alaca_williams2007a} and D.~Ye 
\cite{ye2015}. 

This paper is organized in the following way. 
In \hyperref[modularForms]{Section \ref*{modularForms}} we discuss 
modular forms, briefly define eta functions and convolution sums, and prove  
the generalization of the extraction of the convolution sum.  
Our main results on the evaluation of the convolution sums are discussed in 
\hyperref[convolution_14_22_26]{Section \ref*{convolution_14_22_26}}. 
The determination of formulae for the number of representations of a positive 
integer $n$ is discussed in 
\hyperref[representations_e_f]{Section \ref*{representations_e_f}}. 
A brief outlook is 
given in \hyperref[conclusion]{Section \ref*{conclusion}}.

Software for symbolic scientific computation is used to obtain the results 
of this paper. This software comprises the open source software packages 
\emph{GiNaC}, \emph{Maxima}, \emph{REDUCE}, \emph{SAGE} and the commercial 
software package \emph{MAPLE}.


\section{Modular Forms and Convolution Sums} \label{modularForms}

Let $\mathbb{H}$ be the upper half-plane, that is  
$\mathbb{H}=\{ z\in \mathbb{C}~ | ~\text{Im}(z)>0\}$, 
and let  $G=\text{SL}_{2}(\mathbb{R})$ be the group of $2\times 2$-matrices 
\begin{math}\left(\begin{smallmatrix} a & b \\ c &
 d \end{smallmatrix}\right)\end{math} such that $a,b,c,d\in\mathbb{R}$ and 
$ad-bc=1$ hold. Let furthermore $\Gamma=\text{SL}_{2}(\mathbb{Z})$ be a 
subset of $\text{SL}_{2}(\mathbb{R})$. 
Let $N\in\mathbb{N}$. Then  
\begin{eqnarray*}
\Gamma(N) & = \bigl\{~\left(\begin{smallmatrix} a & b \\ c &
  d \end{smallmatrix}\right)\in\text{SL}_{2}(\mathbb{Z})~ |
  ~\left(\begin{smallmatrix} a & b \\ c &
 d\end{smallmatrix}\right)\equiv\left(\begin{smallmatrix} 1 & 0 \\ 0 & 1 
 \end{smallmatrix}\right) \pmod{N} ~\bigr\}
\end{eqnarray*}
is a subgroup of $G$ and is called a \emph{principal congruence subgroup of 
level N}. A subgroup $H$ of $G$ is called a \emph{congruence subgroup of 
level N} if it contains $\Gamma(N)$.

Relevant for our purposes is the congruence subgroup
 \begin{align*}
\Gamma_{0}(N) & = \bigl\{~\left(\begin{smallmatrix} a & b \\ c &
  d \end{smallmatrix}\right)\in\text{SL}_{2}(\mathbb{Z})~ | ~
   c\equiv 0 \pmod{N} ~\bigr\}.
\end{align*}
Let $k,N\in\mathbb{N}$ and let $\Gamma'\subseteq\Gamma$ be a congruence 
subgroup of level $N\in\mathbb{N}$. 
Let $k\in\mathbb{Z}, \gamma\in\text{SL}_{2}(\mathbb{Z})$ and $f : 
\mathbb{H}\cup\mathbb{Q}\cup\{\infty\} \rightarrow 
\mathbb{C}\cup\{\infty\}$. 
We denote by 
$f^{[\gamma]_{k}}$ the function whose value at $z$ is $(cz+d)^{-
k}f(\gamma(z))$, i.e., $f^{[\gamma]_{k}}(z)=(cz+d)^{-k}f(\gamma(z))$. 
The following definition is according to \nK\ \cite[p. 108]{koblitz-1993}.
\begin{definition} \label{modularForms-defin-2}
Let $N\in\mathbb{N}$, $k\in\mathbb{Z}$, $f$ be a meromorphic function 
on $\mathbb{H}$ and $\Gamma'\subset\Gamma$ a congruence subgroup of 
level $N$. 
\begin{enumerate}
\item[(a)] $f$ is called a \emph{modular function of weight $k$} for 
$\Gamma'$ if
\begin{enumerate}
\item[(a1)] for all $\gamma\in\Gamma'$ it holds that $f^{[\gamma]_{k}}=f$.
\item[(a2)] for any $\delta\in\Gamma$ it holds that $f^{[\delta]_{k}}(z)$ 
can be expressed in the form 
$\underset{n\in\mathbb{Z}}{\sum}a_{n}e^{\frac{2\pi i z n}{N}}$, 
wherein $a_{n}\neq 0$ for finitely many $n\in\mathbb{Z}$ such that $n<0$. 
\end{enumerate}
\item[(b)] $f$ is called a \emph{modular form of weight $k$} for 
$\Gamma'$ if
	\begin{enumerate}
	\item[(b1)] $f$ is a modular function of weight $k$ for $\Gamma'$,
	\item[(b2)] $f$ is holomorphic on $\mathbb{H}$,
	\item[(b3)] for all $\delta\in\Gamma$ and for all $n\in\mathbb{Z}$ 
such that $n<0$ it holds that $a_{n}=0$.
	\end{enumerate}
\item[(c)] $f$ is called a \emph{cusp form of weight $k$ for $\Gamma'$} if
	\begin{enumerate}
	\item[(c1)] $f$ is a modular form of weight $k$ for $\Gamma'$, 
	\item[(c2)] for all $\delta\in\Gamma$ it holds that $a_{0}=0$.
	\end{enumerate}
\end{enumerate}
\end{definition}
For $k,N\in\mathbb{N}$, let $\M_{k}(\Gamma_{0}(N))$ be the space of 
modular forms of weight $k$ for 
$\Gamma_{0}(N)$, $\EuScript{S}_{k}(\Gamma_{0}(N))$ be the subspace of 
cusp forms of weight $k$ for $\Gamma_{0}(N)$, 
and $\E_{k}(\Gamma_{0}(N))$ be the subspace of Eisenstein forms of 
weight $k$ for $\Gamma_{0}(N)$. 
Then it is proved in W.~A.~Stein's book (online version) \cite[p.~81]{wstein} that 
$\M_{k}(\Gamma_{0}(N)) =
\E_{k}(\Gamma_{0}(N))\oplus\EuScript{S}_{k}(\Gamma_{0}(N))$.  

As noted in Section 5.3 of W.~A.~Stein's book \cite[p.~86]{wstein} 
if the primitive Dirichlet 
characters are trivial and $2\leq k$ is even, then 
$E_{k}(q) = 1 - \frac{2k}{B_{k}}\,\underset{n=1}{\overset{\infty}{\sum}}\,
\sigma_{k-1}(n)\,q^{n}$, where $B_{k}$ are the Bernoulli numbers.

For the purpose of this paper we only consider trivial primitive Dirichlet 
characters and $k\geq 4$. Theorems 5.8 and 5.9 in Section 5.3 of 
\cite[p.~86]{wstein} also hold for this special case.

\subsection{Eta Functions}  \label{etaFunctions}

The Dedekind eta function, $\eta(z)$,  
is defined on the upper half-plane $\mathbb{H}$ by
$\eta(z) = e^{\frac{2\pi i z}{24}}\overset{\infty}{\underset{n=1}{\prod}}(1 - e^{2\pi i n z})$.
We set $q=e^{2\pi i z}$. Then 
$\eta(z) = q^{\frac{1}{24}}\overset{\infty}{\underset{n=1}{\prod}}(1 - q^{n}) = q^{\frac{1}{24}} F(q)$,
where $F(q)=\overset{\infty}{\underset{n=1}{\prod}}(1 - q^{n})$.

The Dedekind eta function was systematically used by M.~Newman 
\cite{newman_1957,newman_1959} to construct modular forms for 
$\Gamma_{0}(N)$. M.~Newman then determined when a function $f(z)$ is a 
modular form for $\Gamma_{0}(N)$ by providing conditions (i)-(iv) in 
the following theorem. The order of vanishing of an eta function 
at the cusps of $\Gamma_{0}(N)$, which is condition (v) or (v$'$) in 
the following theorem, was determined by G.~Ligozat \cite{ligozat_1975}.

The following theorem is proved in L.~J.~P.~Kilford's book 
\cite[p.~99]{kilford} and G.~K\"{o}hler's book \cite[p.~37]{koehler}; 
we will apply that theorem to determine eta functions, $f(z)$, which 
belong to $\M_{k}(\Gamma_{0}(N))$, and especially those eta functions 
which are in $\EuScript{S}_{k}(\Gamma_{0}(N))$. 
\begin{theorem}[M.~Newman and G.~Ligozat] \label{ligozat_theorem} 
Let $N\in \mathbb{N}$  and  
$f(z)=\overset{}{\underset{1\leq\delta|N}{\prod}}\eta^{r_{\delta}}(\delta z)$ 
be an eta function which satisfies the following conditions: 

\begin{tabular}{llll}
{\textbf{(i)}} & $\overset{}{\underset{1\leq\delta|N}{\sum}}\delta\,r_{\delta} 
	\equiv 0 \pmod{24}$, & 
{\textbf{(ii)}} &  $\overset{}{\underset{1\leq\delta|N}{\sum}}\frac{N}{\delta} 
	\,r_{\delta} \equiv 0 \pmod{24}$, \\
{\textbf{(iii)}} &  $\overset{}{\underset{1\leq\delta|N}
{\prod}}\delta^{r_{\delta}}$ \text{ is a square in } $\mathbb{Q}$, & 
{\textbf{(iv)}} &  $k=\frac{1}{2}\overset{}{\underset{1\leq\delta|N}
	{\sum}}r_{\delta}$ \text{ is an even integer,} \\
\end{tabular}

{\textbf{(v)}} \text{ for each positive divisor $d$ of $N$ it holds that } 
	$\overset{}{\underset{1\leq\delta|N}{\sum}}\frac{\gcd{(\delta,d)}^{2}}	
	{\delta}\,r_{\delta} \geq 0$. 

Then $f(z)\in\M_{k}(\Gamma_{0}(N))$. 

The eta quotient $f(z)$ belongs to $\EuScript{S}_{k}(\Gamma_{0}(N))$ 
if ${\textbf{(v)}}$ is replaced by
 
{\textbf{(v')}} \text{ for each positive divisor $d$ of $N$ it 
holds that } 
$\overset{}{\underset{1\leq\delta|N}{\sum}}\frac{\gcd{(\delta,d)}^{2}}	
{\delta} r_{\delta} > 0$.
\end{theorem}

\subsection{Convolution Sums $W_{(\alpha, \beta)}(n)$}
\label{convolutionSumsEqns}

Recall that for $\alpha,\beta\in\mathbb{N}$ such that $\alpha\leq\beta$, 
the convolution sum, 
 $W_{(\alpha,\beta)}(n)$, is defined by 
 $W_{(\alpha, \beta)}(n) =  \underset{\substack{
 {(l,m)\in\mathbb{N}_{0}^{2}} \\ {\alpha\,l+\beta\,m=n}
} }{\sum}\sigma(l)\sigma(m)$.

As observed by A.~Alaca et al.\ \cite{alaca_alaca_williams2006}, 
we can assume that $\text{gcd}(\alpha,\beta)=1$.  
Let $q\in\mathbb{C}$ be such that $|q|<1$. Then 
the Eisenstein series $L(q)$ and $M(q)$ are defined as follows:
\begin{align}  
L(q) = E_{2}(q) = 1-24\,\sum_{n=1}^{\infty}\sigma(n)q^{n}, 
\label{evalConvolClass-eqn-3} \\
M(q) = E_{4}(q) = 1 + 240\,\sum_{n=1}^{\infty}\sigma_{3}(n)q^{n}. 
\label{evalConvolClass-eqn-4}
\end{align}
The following two relevant results are essential for the sequel of this 
work and are a generalization of the extraction of the convolution sum 
using Eisenstein forms of weight $4$ for all pairs 
$(\alpha, \beta)\in\mathbb{N}^{2}$. 
\begin{lemma}  \label{evalConvolClass-lema-1}
Let $\alpha, \beta \in \mathbb{N}$. Then 
\begin{equation*}
( \alpha\, L(q^{\alpha}) - \beta\, L(q^{\beta}) )^{2}\in
\M_{4}(\Gamma_{0}(\alpha\beta)).
\end{equation*}
\end{lemma}
\begin{proof} 
If $\alpha=\beta$, then trivially 
$0=(\alpha\, L(q^{\alpha}) - \alpha\,L(q^{\alpha}) )^{2}\in \M_{4}
(\Gamma_{0}(\alpha))$ and there is nothing to prove. 
Therefore, we may suppose that $\alpha\neq\beta>1$ 
in the sequel. 
We apply the result proved by W.~A.~Stein \cite[Thrms 5.8,5.9, p.~86]{wstein} 
to deduce   
$L(q) - \alpha\, L(q^{\alpha})\in \M_{2}(\Gamma_{0}(\alpha))\subseteq 
\M_{2}(\Gamma_{0}(\alpha\beta))$ and  
$L(q) - \beta\, L(q^{\beta}) \in \M_{2}(\Gamma_{0}(\beta))\subseteq
\M_{2}(\Gamma_{0}(\alpha\beta))$. Therefore, 
$$ \alpha\, L(q^{\alpha}) - \beta\,
L(q^{\beta})= (L(q)-\beta\, L(q^{\beta}) ) - (L(q)-\alpha\, L(q^{\alpha})) 
\in \M_{2}(\Gamma_{0}(\alpha\beta))$$ and so 
$(\alpha\, L(q^{\alpha}) - \beta\, L(q^{\beta}) )^{2}\in
\M_{4}(\Gamma_{0}(\alpha\beta))$.
\end{proof}
\begin{theorem} \label{convolutionSum_a_b}
Let $\alpha,\beta\in\mathbb{N}$ be such that
$\alpha$ and $\beta$ are relatively prime and $\alpha < \beta$. 
Then
\begin{multline}
 ( \alpha\, L(q^{\alpha}) - \beta\, L(q^{\beta}) )^{2}  = 
 (\alpha - \beta)^{2} 
    + \sum_{n=1}^{\infty}\biggl(\ 240\,\alpha^{2}\,\sigma_{3}
    (\frac{n}{\alpha}) 
    + 240\,\beta^{2}\,\sigma_{3}(\frac{n}{\beta}) \\ 
    + 48\,\alpha\,(\beta-6n)\,\sigma(\frac{n}{\alpha}) 
    + 48\,\beta\,(\alpha-6n)\,\sigma(\frac{n}{\beta}) 
    - 1152\,\alpha\beta\,W_{(\alpha,\beta)}(n)\,\biggr)q^{n}. 
    \label{evalConvolClass-eqn-11}
\end{multline} 
\end{theorem}
\begin{proof} 
We observe that 
\begin{multline}
( \alpha\, L(q^{\alpha}) - \beta\, L(q^{\beta}) )^{2}  = 
\alpha^{2}\, L^{2}(q^{\alpha}) + \beta^{2}\,
             L^{2}(q^{\beta}) - 2\,\alpha\beta\, 
             L(q^{\alpha})L(q^{\beta}).
             \label{evalConvolClass-eqn-7}
\end{multline}
J.~W.~L.~Glaisher \cite{glaisher} has proved the following identity  
\begin{equation}
L^{2}(q) = 1+\sum_{n=1}^{\infty}\biggl( 240\,\sigma_{3}(n) - 288\, n\,\sigma(n) \biggr)
q^{n} \label{evalConvolClass-eqn-5}
\end{equation}
which we apply to deduce 
\begin{equation}
L^{2}(q^{\alpha}) = 1+\sum_{n=1}^{\infty}\bigl( 240\,\sigma_{3}(\frac{n}{\alpha}) -
288\,\frac{n}{\alpha}\,\sigma(\frac{n}{\alpha})\bigr) q^{n} \label{evalConvolClass-eqn-8}
\end{equation}
and 
\begin{equation}
L^{2}(q^{\beta}) = 1+\sum_{n=1}^{\infty}\bigl( 240\,\sigma_{3}(\frac{n}{\beta}) -
288\,\frac{n}{\beta}\,\sigma(\frac{n}{\beta})\bigr)q^{n}. \label{evalConvolClass-eqn-9}
\end{equation}
Since  
\begin{multline*}
\bigl(\sum_{n=1}^{\infty}\sigma(\frac{n}{\alpha})q^{n}\bigr)\bigl(\sum_{n=1}^{\infty}\sigma(\frac{n}{\beta})q^{n}\bigr)
  =
\sum_{n=1}^{\infty}\bigl(\sum_{\alpha k + \beta l=n}\sigma(k)\sigma(l)\,\bigr)q^{n}
  = \sum_{n=1}^{\infty}W_{(\alpha,\beta)}(n)q^{n},
\end{multline*}
we conclude, when using the accordingly modified  
\autoref{evalConvolClass-eqn-3}, that 
\begin{equation}
L(q^{\alpha})L(q^{\beta}) = 1 - 24\,\sum_{n=1}^{\infty}\sigma(\frac{n}{\alpha})q^{n} -
24\,\sum_{n=1}^{\infty}\sigma(\frac{n}{\beta})q^{n} + 576\,\sum_{n=1}^{\infty}W_{(\alpha,\beta)}(n)q^{n}.\label{evalConvolClass-eqn-10}
\end{equation}
Therefore, 
\begin{multline*}
\bigl( \alpha\, L(q^{\alpha}) - \beta\, L(q^{\beta})\bigr)^{2} 
   = (\alpha - \beta)^{2} + \sum_{n=1}^{\infty}\biggl(\ 240\,\alpha^{2}\,\sigma_{3}(\frac{n}{\alpha}) 
    + 240\,\beta^{2}\,\sigma_{3}(\frac{n}{\beta}) \\ 
    +  48\,\alpha\,(\beta-6n)\,\sigma(\frac{n}{\alpha})  
       + 48\,\beta\,(\alpha-6n)\,\sigma(\frac{n}{\beta}) 
     - 1152\,\alpha\beta\, W_{(\alpha,\beta)}(n)\ \biggr)q^{n}  
\end{multline*}
as asserted. 
\end{proof}


\section{Evaluation of the convolution sums $W_{(\alpha,\beta)}(n)$ for $\alpha\beta=14,22,26$}
\label{convolution_14_22_26}

In this section, we give explicit formulae for the convolution sums 
$W_{(2,7)}(n)$, $W_{(1,22)}(n)$, $W_{(2,11)}(n)$, $W_{(1,26)}(n)$ and $W_{(2,13)}(n)$. 
Note that an explicit formula for the convolution sum $W_{(1,14)}(n)$ is 
proved by E.~Royer \cite{royer}.

\subsection{Bases for $\E_{4}(\Gamma_{0}(\alpha\beta))$ and
  $\EuScript{S}_{4}(\Gamma_{0}(\alpha\beta))$  for $\alpha\beta=14,22,26$}  
\label{convolution_14_22_26-gen}

We use the dimension formulae for the space of Eisenstein forms and 
the space of cusp forms in T.~Miyake's book  
\cite[Thrm 2.5.2,~p.~60]{miyake1989} or W.~A.~Stein's book  
\cite[Prop.\ 6.1, p.\ 91]{wstein} to deduce that   
$\text{dim}(\E_{4}(\Gamma_{0}(14)))=\text{dim}(\E_{4}(\Gamma_{0}(22)))=
\text{dim}(\E_{4}(\Gamma_{0}(26)))=4$, 
$\text{dim}(\EuScript{S}_{4}(\Gamma_{0}(14))=4$,   
$\text{dim}(\EuScript{S}_{4}(\Gamma_{0}(22))=7$  and  
$\text{dim}(\EuScript{S}_{4}(\Gamma_{0}(26))=9$.   
By \autoref{ligozat_theorem} the following eta functions 
\begin{itemize}
  \item $A_{i}(q)$, $1\leq i\leq 4$, are elements of 
$\EuScript{S}_{4}(\Gamma_{0}(14))$.
\begin{longtable}{rclcrcl}
$A_{1}(q)$ & = & $\frac{\eta^{5}(z)\eta^{5}(7z)}{\eta(2z)\eta(14z)}$ & ~ &
$A_{2}(q)$ & = & $\eta^{2}(z)\eta^{2}(2z)\eta^{2}(7z)\eta^{2}(14z)$ \\
 ~ \\
$A_{3}(q)$ & = & $\frac{\eta^{5}(2z)\eta^{5}(14z)}{\eta(z)\eta(7z)}$  & ~ &
$A_{4}(q)$ & = &  $\frac{\eta^{6}(z)\eta^{6}(14z)}{\eta^{2}(2z)\eta^{2}(7z)}$ 
\end{longtable}
\item $B_{i}(q), ~1\leq i\leq 7$ are elements of 
$\EuScript{S}_{4}(\Gamma_{0}(22))$. 
\begin{longtable}{rclcrcl}
$B_{1}(q)$ & = & $\frac{\eta^{6}(z)\eta^{6}(11z)}{\eta^{2}(2z)\eta^{2}(22z)}$ & ~ &
$B_{2}(q)$ & = & $\eta^{4}(z)\eta^{4}(11z)$ \\ ~ \\
$B_{3}(q)$ & = & $\eta^{2}(z)\eta^{2}(2z)\eta^{2}(11z)\eta^{2}(22z)$ \\
$B_{4}(q)$ & = & $\eta^{4}(2z)\eta^{4}(22z)$   & ~ &
$B_{5}(q)$ & = &  $\frac{\eta^{6}(2z)\eta^{6}(22z)}{\eta^{2}(z)\eta^{2}(11z)}$
\\  ~ \\
$B_{6}(q)$ & = & $\frac{\eta(2z)\eta^{3}(11z)\eta^{5}(22z)}{\eta(z)}$  & ~ &
$B_{7}(q)$ & = &  $\frac{\eta^{9}(2z)\eta^{7}(11z)}{\eta^{5}(z)\eta^{3}(22z)}$ 
\end{longtable}
\item $C_{i}(q), ~1\leq i\leq 9$, are elements of 
$\EuScript{S}_{4}(\Gamma_{0}(26))$.
\begin{longtable}{rclcrcl}
$C_{1}(q)$ & = & $\frac{\eta(z)\eta^{5}(2z)\eta^{3}(13z)}{\eta(26z)}$ & ~ &
$C_{2}(q)$ & = & $\eta^{3}(z)\eta^{3}(2z)\eta(13z)\eta(26z)$ \\ ~ \\
$C_{3}(q)$ & = & $\eta(z)\eta^{3}(2z)\eta^{3}(13z)\eta(26z)$   & ~ &
$C_{4}(q)$ & = & $\eta^{3}(z)\eta(2z)\eta(13z)\eta^{3}(26z)$  \\  ~ \\
$C_{5}(q)$ & = &  $\eta(z)\eta(2z)\eta^{3}(13z)\eta^{6}(26z)$   & ~ &
$C_{6}(q)$ & = & $\frac{\eta^{3}(z)\eta(13z)\eta^{5}(26z)}{\eta(2z)}$  \\ ~ \\
$C_{7}(q)$ & = &  $\frac{\eta^{3}(2z)\eta^{5}(13z)\eta(26z)}{\eta(z)}$ & ~ & 
$C_{8}(q)$ & = & $\frac{\eta^{5}(2z)\eta^{5}(13z)}{\eta(z)\eta(26z)}$  \\ ~ \\
$C_{9}(q)$ & = &  $\frac{\eta^{7}(z)\eta^{7}(26z)}{\eta^{3}(2z)\eta^{3}(13z)}$ 
\end{longtable}
\end{itemize}
The eta functions 
\begin{itemize}
\item $A_{i}(q), ~1\leq i\leq 4$, can be expressed in the form 
$\underset{n=1}{\overset{\infty}{\sum}}a_{i}(n)q^{n}$; 
\item $B_{i}(q), ~1\leq i\leq 7$, can be expressed in the form 
$\underset{n=1}{\overset{\infty}{\sum}}b_{i}(n)q^{n}$; and  
\item $C_{i}(q), ~1\leq i\leq 9$, can be expressed in the form 
$\underset{n=1}{\overset{\infty}{\sum}}c_{i}(n)q^{n}$. 
\end{itemize}

\begin{theorem} \label{basisCusp-14_22_26}
\begin{enumerate}
\item[\textbf{(a)}] The sets 
\begin{align*}
B_{E,14} & =\{\, M(q^{t})\,\mid\, t\text{ is a positive divisor of } 14\,\}, \\ 
B_{E,22} & =\{\, M(q^{t})\,\mid\, t\text{ is a positive divisor of } 22\,\}, \\
B_{E,26} & =\{\, M(q^{t})\,\mid\, t\text{ is a positive divisor of } 26\,\} 
\end{align*}
constitute bases of 
$\E_{4}(\Gamma_{0}(14))$, $\E_{4}(\Gamma_{0}(22))$ and
$\E_{4}(\Gamma_{0}(26))$, respectively.
\item[\textbf{(b)}] The sets 
$B_{S,14}=\{\, A_{i}(q)~\mid ~1\leq i\leq 4\,\}$, 
$B_{S,22}=\{\, B_{i}(q)~\mid ~1\leq i\leq 7\,\}$ and  
$B_{S,26}=\{\, C_{i}(q)~\mid ~1\leq i\leq 9\,\}$ 
are bases of $\EuScript{S}_{4}(\Gamma_{0}(14))$, 
$\EuScript{S}_{4}(\Gamma_{0}(22))$ and $\EuScript{S}_{4}(\Gamma_{0}(26))$, 
respectively.
\item[\textbf{(c)}] The sets $B_{M,14}=B_{E,14}\cup B_{S,14}$, 
$B_{M,22}=B_{E,22}\cup B_{S,22}$ and $B_{M,26}=B_{E,26}\cup B_{S,26}$ 
constitute bases of $\M_{4}(\Gamma_{0}(14))$, $\M_{4}(\Gamma_{0}(22))$ and 
$\M_{4}(\Gamma_{0}(26))$, respectively. 
\end{enumerate}
\end{theorem}
\begin{proof} We only give the proof for the case related to $14$ 
since the other two cases are proved similarly.
\begin{enumerate}
\item[\textbf{(a)}] By Theorem 5.8 in Section 5.3 of \cite[p.~86]{wstein}
each $M(q^{t})$ is in $\M_{4}(\Gamma_{0}(t))$, where $t$ is a positive divisor
of $14$. Since the dimension of $\E_{4}(\Gamma_{0}(14))$ is finite, it suffices 
to show that $M(q^{t})$ with $t|14$ are linearly independent. 
Suppose that $x_{1},x_{2},x_{7},x_{14}\in\mathbb{C}$ and  
$\underset{\delta|14}{\sum} x_{\delta}\,M(q^{\delta})=0$. That is, 
\begin{equation*}
\underset{\delta|14}{\sum}
x_{\delta}+ 240\,\sum_{n=1}^{\infty}\bigl(\underset{\delta|14}{\sum}
x_{\delta}\sigma_{3}(\frac{n}{\delta})\bigr)q^{n} = 0.
\end{equation*}
We then equate the coefficients of $q^{n}$ for $n=1,2,7,14$ to obtain 
the following system of linear equations 
\begin{align*}
x_{1} & = 0 \\
9\,x_{1} + x_{2} & = 0 \\
344\,x_{1} + x_{7} & = 0 \\
3096\,x_{1} + 344\,x_{2} + 9\,x_{7} + x_{14} & = 0
\end{align*}
whose unique solution is $x_{1}=x_{2}=x_{7}=x_{14}=0$. So, the set $B_{E,14}$ 
is linearly independent. 
Hence, the set $B_{E,14}$ is a basis of $\E_{4}(\Gamma_{0}(14))$.
\item[\textbf{(b)}]  We first show that each $A_{i}(q)$, where $1\leq i\leq 4$,
is in the space $\EuScript{S}_{4}(\Gamma_{0}(14))$. 
That is implicit since $A_{i}(q)$ with $1\leq i\leq 4$ are obtained 
from an exhaustive search using \autoref{ligozat_theorem} $(i)-(v')$. 
Since the dimension of $\EuScript{S}_{4}(\Gamma_{0}(14))$ is $4$, it 
suffices to show that the set $\{\,A_{i}(q)~\mid~ 1\leq i\leq 4\,\}$ 
is linearly independent.
Suppose that $x_{1},x_{2},x_{3},x_{4}\in\mathbb{C}$ and 
\begin{equation*}
x_{1}\,A_{1}(q)+x_{2}\,A_{2}(q)+x_{3}\,A_{3}(q)+x_{4}\,A_{4}(q)=0.
\end{equation*}
Then   
\begin{equation*}
\underset{n=1}{\overset{\infty}{\sum}}(\,x_{1}\,a_{1}(n)+x_{2}\,a_{2}(n) 
+ x_{3}\,a_{3}(n) + x_{4}\,a_{4}(n)\,)q^{n} = 0.
\end{equation*}
So, when we equate the coefficients of $q^{n}$ for $n=1,2,3,4$, we obtain  
the following system of linear equations 
\begin{align*}
x_{1} - 5\,x_{2} + 6\,x_{3} + 5\,x_{4} & = 0 \\
x_{2} - 2\,x_{3} - 3\,x_{4}   & = 0 \\
x_{3} + x_{4} & = 0 \\
x_{3} - 6\,x_{4} & = 0
\end{align*}
whose unique solution is $x_{1}=x_{2}=x_{3}=x_{4}=0$. 
So, the set $B_{S,14}$ 
is linearly independent. 
Hence, the set $B_{S,14}$ is a basis of $\EuScript{S}_{4}(\Gamma_{0}(14))$.
\item[\textbf{(c)}]
Since $\M_{4}(\Gamma_{0}(14))=\E_{4}(\Gamma_{0}(14))\oplus 
\EuScript{S}_{4}(\Gamma_{0}(14))$, the result follows from (a) and (b).
\end{enumerate}
\end{proof}

\subsection{Evaluation of $W_{(\alpha,\beta)}(n)$ for $(\alpha,\beta)=(2,7),(1,22),(2,11),(1,26),(2,13)$} 
\label{convolSum-w_14_22_26}  

\begin{lemma} \label{lema-14_22_26}
We have 
\begin{multline}
(2\, L(q^{2}) - 7\, L(q^{7}))^{2} 
 = 25 + \sum_{n=1}^{\infty}\biggl(\, - \frac{672}{25}\,\sigma_{3}(n) 
	+ \frac{21312}{25}\,\sigma_{3}(\frac{n}{2})  
   + \frac{261072}{25}\,\sigma_{3}(\frac{n}{7})  \\
    - \frac{131712}{25}\,\sigma_{3}(\frac{n}{14}) 
   + \frac{672}{25}\,a_{1}(n)  
   + \frac{96}{25}\,a_{2}(n) 
  + \frac{5376}{25}\,a_{3}(n) 
  + 384\,a_{4}(n)\, \biggr)q^{n},
\label{convolSum-eqn-1}  \\
\end{multline}
\begin{multline}
(2\, L(q^{2}) - 11\, L(q^{11}))^{2} 
   = 81 + \sum_{n=1}^{\infty}\biggl(\,\frac{15840}{61}\,\sigma_{3}(n) 
	+ \frac{37440}{61}\,\sigma_{3}(\frac{n}{2}) \\
   + \frac{1626768}{61}\,\sigma_{3}(\frac{n}{11}) 
  - \frac{494208}{61}\,\sigma_{3}(\frac{n}{22}) 
    + \frac{36864}{61}\,b_{1}(n) 
    + \frac{357408}{61}\,b_{2}(n)  \\ 
    + \frac{1160352}{61}\,b_{3}(n) 
    + \frac{1539072}{61}\,b_{4}(n)   
    + \frac{834048}{61}\,b_{5}(n)   
   - 22176\,b_{6}(n)  
   - 864\,b_{7}(n) \,\biggr) q^{n}, 
\label{convolSum-eqn-2-11}
\end{multline}
\begin{multline}
(L(q) - 26\, L(q^{26}))^{2} 
   = 625 + \sum_{n=1}^{\infty}\biggl(\,\frac{19152}{85}\,\sigma_{3}(n) 
	- \frac{4992}{85}\,\sigma_{3}(\frac{n}{2}) \\
    - \frac{210912}{85}\,\sigma_{3}(\frac{n}{13})   
    + \frac{12946752}{85}\,\sigma_{3}(\frac{n}{26}) 
    + \frac{82848}{85}\,c_{1}(n)  
    - \frac{4128}{17}\,c_{2}(n) 
   + \frac{61920}{17}\,c_{3}(n) \\
   - \frac{177216}{85}\,c_{4}(n) 
    - \frac{53664}{17}\,c_{5}(n)  
     + \frac{1077024}{85}\,c_{7}(n) 
    + \frac{291072}{85}\,c_{8}(n)  
    - \frac{1248}{85}\,c_{9}(n)\, \biggr) q^{n}, 
\label{convolSum-eqn-1-26-1}
\end{multline}
\begin{multline}
(2\, L(q^{2}) - 13\, L(q^{13}))^{2} 
   = 121 + \sum_{n=1}^{\infty}\biggl(\, -\frac{1248}{85}\,\sigma_{3}(n) 
	+ \frac{76608}{85}\,\sigma_{3}(\frac{n}{2}) \\
   + \frac{3236688}{85}\,\sigma_{3}(\frac{n}{13}) 
  - \frac{843648}{85}\,\sigma_{3}(\frac{n}{26}) 
    + \frac{1248}{85}\,c_{1}(n) 
    + \frac{12192}{17}\,c_{2}(n) 
    + \frac{52128}{17}\,c_{3}(n) \\
    + \frac{181824}{85}\,c_{4}(n)   
   + \frac{158496}{17}\,c_{5}(n) 
   + \frac{16224}{85}\,c_{7}(n) 
    - \frac{35328}{85}\,c_{8}(n) 
    - \frac{82848}{85}\,c_{9}(n)\, \biggr) q^{n}. 
\label{convolSum-eqn-2-13}
\end{multline}
\end{lemma}
\begin{proof} Since the other cases are proved similarly, we only give 
the proof for $(2\, L(q^{2}) - 7\, L(q^{7}))^{2}$.

We apply \hyperref[evalConvolClass-lema-1]{Lemma \ref*{evalConvolClass-lema-1}} 
with $\alpha=2$ and $\beta=7$ and we use \autoref{basisCusp-14_22_26} (c) to 
infer that there
exist $x_{1},x_{2},x_{7},x_{14},y_{1},y_{2},y_{3},y_{4}\in\mathbb{C}$ such
that 
\begin{align}
( 2 L(q^{2}) - 7 L(q^{7}) )^{2} &  = \sum_{\delta|14}x_{\delta}M(q^{\delta}) + 
\sum_{j=1}^{4}y_{j}A_{j}(q)  \notag \\ &
  = \sum_{\delta|14}x_{\delta} + \sum_{i=1}^{\infty}\biggl(\,240\,\sum_{\delta|14}x_{\delta}
  \sigma_{3}(\frac{n}{\delta}) + \sum_{j=1}^{4}y_{j}a_{j}(n)\,\biggr)\,q^{n}. 
      \label{convolution_2_7-eqn-0}
\end{align} 
Now, when we equate the right hand side of 
\autoref{convolution_2_7-eqn-0} with that of 
\autoref{evalConvolClass-eqn-11}, and when we take the 
coefficients of $q^{n}$ for which $n= 1,2,3,4,5,7,9,14$ for example, we 
obtain a system of linear equations whose solution is unique. 
Hence, we obtain the stated result.
\end{proof} 
Now we state and prove our main result of this Subsection. 
\begin{theorem} \label{convolSum-theor-w_14_22_26}
Let $n$ be a positive integer. Then 
\begin{align*}
 W_{(2,7)}(n)  = &  \frac{1}{600}\,\sigma_{3}(n) 
   + \frac{1}{150}\,\sigma_{3}(\frac{n}{2}) 
    + \frac{49}{600}\,\sigma_{3}(\frac{n}{7}) 
    + \frac{49}{150}\,\sigma_{3}(\frac{n}{14}) 
    + (\frac{1}{24}-\frac{1}{28}\,n)\,\sigma(\frac{n}{2}) \\ &
   + (\frac{1}{24}-\frac{1}{8}\,n)\,\sigma(\frac{n}{7})    
   - \frac{1}{600}\,a_{1}(n) 
   - \frac{1}{4200}\,a_{2}(n) 
  - \frac{1}{75}\,a_{3}(n) 
  - \frac{1}{42}\,a_{4}(n), \\ 
\end{align*}
\begin{align*}
  W_{(1,22)}(n)   = & \frac{17}{1464}\,\sigma_{3}(n) 
   - \frac{1}{122}\,\sigma_{3}(\frac{n}{2})
   + \frac{35}{488}\,\sigma_{3}(\frac{n}{11}) 
   +  \frac{125}{366}\,\sigma_{3}(\frac{n}{22}) \\ &
   + (\frac{1}{24}-\frac{1}{88}\,n)\,\sigma(n) 
   + (\frac{1}{24}-\frac{1}{4}\,n)\,\sigma(\frac{n}{22}) 
   - \frac{21}{2684}\,b_{1}(n)   
   - \frac{159}{5368}\,b_{2}(n)  \\ & 
   - \frac{69}{5368}\,b_{3}(n) 
  - \frac{32}{671}\,b_{4}(n) 
  + \frac{2}{61}\,b_{5}(n) 
  - \frac{7}{8}\,b_{6}(n) 
  - \frac{3}{88}\,b_{7}(n), \\
\end{align*}
\begin{align*}
 W_{(2,11)}(n)  = &  - \frac{5}{488}\,\sigma_{3}(n)
  + \frac{5}{366}\,\sigma_{3}(\frac{n}{2}) 
  + \frac{137}{1464}\,\sigma_{3}(\frac{n}{11}) 
  + \frac{39}{122}\,\sigma_{3}(\frac{n}{22})   \\ &
  + (\frac{1}{24}-\frac{1}{44}\,n)\,\sigma(\frac{n}{2}) 
  + (\frac{1}{24}-\frac{1}{8}\,n)\,\sigma(\frac{n}{11})  
  - \frac{16}{671}\,b_{1}(n)  
  - \frac{1241}{5368}\,b_{2}(n) \\ & 
  - \frac{4029}{5368}\,b_{3}(n) 
  - \frac{668}{671}\,b_{4}(n) 
   - \frac{362}{671}\,b_{5}(n)    
   + \frac{7}{8}\,b_{6}(n)    
  + \frac{3}{88}\,b_{7}(n),  \\
\end{align*}
\begin{align*}
  W_{(1,26)}(n)  = & \frac{1}{2040}\,\sigma_{3}(n) 
  + \frac{1}{510}\,\sigma_{3}(\frac{n}{2})
  + \frac{169}{2040}\,\sigma_{3}(\frac{n}{13}) 
    +  \frac{169}{510}\,\sigma_{3}(\frac{n}{26}) \\ &
   + (\frac{1}{24}-\frac{1}{104}\,n)\,\sigma(n) 
   + (\frac{1}{24}-\frac{1}{4}\,n)\,\sigma(\frac{n}{26}) 
   - \frac{863}{26520}\,c_{1}(n)   
   + \frac{43}{5304}\,c_{2}(n)   \\ &
   - \frac{215}{1768}\,c_{3}(n)
  + \frac{71}{1020}\,c_{4}(n) 
  + \frac{43}{408}\,c_{5}(n) 
  - \frac{863}{2040}\,c_{7}(n) 
  - \frac{379}{3315}\,c_{8}(n)   \\ &
  + \frac{1}{2040}\,c_{9}(n),  \\
\end{align*}
\begin{align*}
 W_{(2,13)}(n)  = & \frac{1}{2040}\,\sigma_{3}(n)
  + \frac{1}{510}\,\sigma_{3}(\frac{n}{2}) 
  + \frac{169}{2040}\,\sigma_{3}(\frac{n}{13}) 
  + \frac{169}{510}\,\sigma_{3}(\frac{n}{26})   \\ &
  + (\frac{1}{24}-\frac{1}{52}\,n)\,\sigma(\frac{n}{2}) 
  + (\frac{1}{24}-\frac{1}{8}\,n)\,\sigma(\frac{n}{13})  
  - \frac{1}{2040}\,c_{1}(n)  
  - \frac{127}{5304}\,c_{2}(n) \\ &
- \frac{181}{1768}\,c_{3}(n) 
  - \frac{947}{13260}\,c_{4}(n) 
   - \frac{127}{408}\,c_{5}(n)    
  - \frac{13}{2040}\,c_{7}(n)  
  + \frac{46}{3315}\,c_{8}(n) \\ &
+ \frac{863}{26520}\,c_{9}(n). 
\end{align*}
\end{theorem}
 \begin{proof} We give the proof for the $W_{(2,7)}(n)$ since the other  
cases are proved similarly.

It follows immediately when we set 
$\alpha=2$ and $\beta=7$ in the right hand side of  
\autoref{evalConvolClass-eqn-11}, equate the so-obtained result with 
the right hand side of \autoref{convolSum-eqn-1} and solve for 
$W_{(2,7)}(n)$.
\end{proof}

\section{Number of Representations of a positive Integer  $n$  by the 
Octonary Quadratic Form using $W_{(1,4)}(n) ,W_{(3,4)}(n),W_{(3,8)}(n)$ 
and $W_{(4,9)}(n)$  
}
\label{representations_e_f}

The following number of representations of a positive integer $n$ are 
applications of the results of the evaluation of some convolution sums 
by J.~G.~Huard et al.\ \cite{huardetal}, 
A.~Alaca et al.\ \cite{alaca_alaca_williams2006,alaca_alaca_williams2007a} 
and D.~Ye \cite{ye2015}.

Let $n\in\mathbb{N}_{0}$ and the number of representations
of $n$ by the quaternary quadratic form  $x_{1}^{2} +x_{2}^{2}+x_{3}^{2} +
x_{4}^{2}$ be denoted by $r_{4}(n)$. That means, 
\begin{equation*}
r_{4}(n)=\text{card}(\{(x_{1},x_{2},x_{3},x_{4})\in\mathbb{Z}^{4}~|~ m = x_{1}^{2} +x_{2}^{2} + x_{3}^{2} + x_{4}^{2}\}).
\end{equation*}
We set $r_{4}(0) = 1$. For all $n\in\mathbb{N}$, 
the following Jacobi's identity is proved in \ksW ' book 
\cite[Thrm 9.5, p.\ 83]{williams2011} 
\begin{equation}
r_{4}(n) = 8\sigma(n) - 32\sigma(\frac{n}{4}). \label{representations-eqn-4-1}
\end{equation}

Let furthermore the number of representations of $n$ by the octonary quadratic form 
\begin{equation*}
a\,(x_{1}^{2} +x_{2}^{2} + x_{3}^{2} + x_{4}^{2})
+ b\,(x_{5}^{2} + x_{6}^{2} + x_{7}^{2} + x_{8}^{2})
\end{equation*}
be denoted by $N_{(a,b)}(n)$. That means,  
\begin{multline*}
N_{(a,b)}(n) 
=\text{card}
(\{(x_{1},x_{2},x_{3},x_{4},x_{5},x_{6},x_{7},x_{8})\in\mathbb{Z}^{8}~|~
n = a\,( x_{1}^{2} +x_{2}^{2}  \\
    + x_{3}^{2} + x_{4}^{2} ) + 
b\,( x_{5}^{2} +x_{6}^{2} + x_{7}^{2} + x_{8}^{2}) \}).
\end{multline*}

We then infer the following result:
\begin{theorem} \label{representations-thrm_3_4}
Let $n\in\mathbb{N}$ and $(a,b)=(1,1),(1,3),(2,3),(1,9)$. Then  
\begin{align*}
N_{(1,1)}(n) & = 16\,\sigma(n) - 64\,\sigma(\frac{n}{4})  
 + 64\, W_{(1,1)}(n) + 1024\, W_{(1,1)}(\frac{n}{4}) - 512\,W_{(1,4)}(n) \\ & 
   = 16\,\sigma_{3}(n) - 32\,\sigma_{3}(\frac{n}{2}) + 256\,\sigma_{3}(\frac{n}{4}), \\
N_{(1,3)}(n) & = 8\sigma(n) - 32\sigma(\frac{n}{4}) + 8\sigma(\frac{n}{3}) -
32\sigma(\frac{n}{12}) \\ &
 + 64\, W_{(1,3)}(n) + 1024\, W_{(1,3)}(\frac{n}{4}) 
 - 256\, \biggl( W_{(3,4)}(n) + W_{(1,12)}(n) \biggr), \\
N_{(2,3)}(n)  & = 8\,\sigma(\frac{n}{2}) - 32\,\sigma(\frac{n}{8}) 
+ 8\,\sigma(\frac{n}{3}) - 32\,\sigma(\frac{n}{12}) \\ &
 + 64\, W_{(1,3)}(n) + 1024\, W_{(1,3)}(\frac{n}{4}) 
 - 256\, \biggl( W_{(3,8)}(n) + W_{(1,12)}(n) \biggr), \\
N_{(1,9)}(n)  & = 8\,\sigma(n) - 32\,\sigma(\frac{n}{4}) 
+ 8\,\sigma(\frac{n}{9}) - 32\,\sigma(\frac{n}{36}) \\ & 
 + 64\,W_{(1,9)}(n) + 1024\,W_{(1,9)}(\frac{n}{4}) 
 - 256\,\biggl( W_{(4,9)}(n) + W_{(1,36)}(n) \biggr). 
\end{align*}
\end{theorem}
\begin{proof} We only prove the case $N_{(1,3)}(n)$ since the other cases are proved 
similarly.

It holds that 
\begin{multline*}
N_{(1,3)}(n)  = \sum_{\substack{
{(l,m)\in\mathbb{N}_{0}^{2}} \\ {l+3\,m=n}
 }}r_{4}(l)r_{4}(m) 
   = r_{4}(n)r_{4}(0) + r_{4}(0)r_{4}(\frac{n}{3})  
   + \sum_{\substack{
{(l,m)\in\mathbb{N}^{2}} \\ {l+3\,m=n}
 }}r_{4}(l)r_{4}(m)
\end{multline*}
We make use of \autoref{representations-eqn-4-1} to derive 
\begin{multline*}
N_{(1,3)}(n)  = 8\sigma(n) - 32\sigma(\frac{n}{4}) + 8\sigma(\frac{n}{3}) -
32\sigma(\frac{n}{12}) \\
   + \sum_{\substack{
{(l,m)\in\mathbb{N}^{2}} \\ {l+3m=n}
  }} (8\sigma(l) - 32\sigma(\frac{l}{4}))(8\sigma(m) - 32\sigma(\frac{m}{4})). 
\end{multline*}
We observe that 
\begin{multline*}
(8\sigma(l) - 32\sigma(\frac{l}{4}))(8\sigma(m) - 32\sigma(\frac{m}{4}))  = 
64\sigma(l)\sigma(m) - 256\sigma(\frac{l}{4})\sigma(m) \\
   - 256\sigma(l)\sigma(\frac{m}{4})  + 1024\sigma(\frac{l}
   {4})\sigma(\frac{m}{4}).
\end{multline*}
The evaluation of 
\begin{equation*}
W_{(1,3)}(n) = \sum_{\substack{
{(l,m)\in\mathbb{N}^{2}} \\ {l+3m=n}
 }}\sigma(l)\sigma(m)
\end{equation*}
is shown by J.~G.~Huard et al.\ \cite{huardetal}. We map $l$ to $4l$ to infer  
\begin{equation*}
W_{(4,3)}(n) = \sum_{\substack{
{(l,m)\in\mathbb{N}^{2}} \\ {l+3m=n}
}}\sigma(\frac{l}{4})\sigma(m) 
 = \sum_{\substack{
{(l,m)\in\mathbb{N}^{2}} \\ {4\,l+3m=n}
 }}\sigma(l)\sigma(m).
\end{equation*}
The evaluation of $W_{(4,3)}(n)=W_{(3,4)}(n)$ is proved by 
A.~Alaca et al.\ \cite{alaca_alaca_williams2006}. We next map 
$m$ to $4m$ to conclude  
\begin{equation*}
W_{(1,12)}(n) = \sum_{\substack{
{(l,m)\in\mathbb{N}^{2}} \\ {l+3m=n}
}}\sigma(l)\sigma(\frac{m}{4})  = \sum_{\substack{
{(l,m)\in\mathbb{N}^{2}} \\ {l+12\,m=n}
 }}\sigma(l)\sigma(m).
\end{equation*}
A.~Alaca et al.\ \cite{alaca_alaca_williams2006} have shown the 
evaluation of $W_{(1,12)}(n)$. 
We simultaneously map $l$ to $4l$ and $m$ to $4m$ to deduce 
\begin{equation*}
\sum_{\substack{
{(l,m)\in\mathbb{N}^{2}} \\ {l+3m=n}
}}\sigma(\frac{l}{4})\sigma(\frac{m}{4}) 
= \sum_{\substack{
{(l,m)\in\mathbb{N}^{2}} \\ {l+3m=\frac{n}{4}}
 }}\sigma(l)\sigma(m)
 = W_{(1,3)}(\frac{n}{4})
\end{equation*}
J.~G.~Huard et al.\ \cite{huardetal} have proved the evaluation 
of $W_{(1,3)}(n)$. 

We then put these evaluations together to obtain the stated result 
for $N_{(1,3)}(n)$. 
\end{proof}

\section{Concluding Remark and future Work} \label{conclusion}

As displayed on \autoref{introduction-table-1}, convolution sums are 
so far evaluated individually, i.e., there is no evaluation of the convolution
sums for a class of positive integers. Since convolution sums are used to
determine explicit formulae for the number of representations of a positive 
integer $n$ by the octonary quadratic forms \autoref{introduction-eq-1} and 
\autoref{introduction-eq-2}, respectively, there is no explicit formulae for 
the number of representations for a class of positive integers by the 
octonary quadratic forms as well. This is a work in progress.

\proof[Acknowledgments]  
The research of the first two authors was supported 
by Discovery Grants from the Natural Sciences and Engineering Research 
Council of Canada (RGPIN-418029-2013 and RGPIN-2015-05208). 



\begin{thebibliography}{10}
\bibitem{glaisher}
J.~W.~L. Glaisher.
\newblock On the square of the series in which the coefficients are the sums of
  the divisors of the exponents.
\newblock {\em Messenger Math}, 14:156--163, 1862.
\end{thebibliography}


\section*{Tables}

\begin{longtable}{|r|r|r|} \hline 
\textbf{$(\alpha,\beta)$}  &  \textbf{Authors}   &  \textbf{References}  \\ \hline
(1,1)  &  M.~Besge, J.~W.~L.~Glaisher, & ~ \\ 
 ~  & S.~Ramanujan  & \cite{besge,glaisher,ramanujan} \\ \hline
(1,2),(1,3),(1,4)  & J.~G.~Huard \& Z.~M.~Ou & ~ \\ 
 ~ &  \& B.~K.~Spearman \& K.~S.~Williams   & \cite{huardetal} \\ \hline
(1,5),(1,7)  & M.~Lemire \& K.~S.~Williams, & ~ \\ 
  ~  &  S.~Cooper \& P.~C.~Toh   & \cite{lemire_williams,cooper_toh} \\ \hline
(1,6),(2,3)  & S.~Alaca \& K.~S.~Williams   & \cite{alaca_williams} \\ \hline
(1,8), (1,9)  & K.~S.~Williams   & \cite{williams2006, williams2005} \\ \hline
(1,10), (1,11),(1,13), &  ~ & ~ \\ 
 (1,14)  & E.~Royer   & \cite{royer} \\ \hline
(1,12),(1,16),(1,18), &  ~ & ~ \\ 
(1,24),(2,9),(3,4), & A.~Alaca \& S.~Alaca \& K.~S.~Williams   &
\cite{alaca_alaca_williams2006,alaca_alaca_williams2007,alaca_alaca_williams2007a,alaca_alaca_williams2008}
\\ 
 (3,8)  &  ~ & ~ \\ \hline
(1,15),(3,5)  & B.~Ramakrishman \& B.~Sahu   & \cite{ramakrishnan_sahu} \\ \hline
(1,20),(2,5),(4,5)  & S.~Cooper \& D.~Ye   & \cite{cooper_ye2014} \\ \hline
(1,23)  & H.~H.~Chan \& S.~Cooper   & \cite{chan_cooper2008} \\ \hline
(1,25)  & E.~X.~W.~Xia \& X.~L.~Tian & ~ \\ 
 ~  &  \& O.~X.~M.~Yao   & \cite{xiaetal2014} \\ \hline
(1,27),(1,32)  & S.~Alaca \& Y.~Kesicio$\check{g}$lu   & \cite{alaca_kesicioglu2014} \\ \hline
(1,36),(4,9)  & D.~Ye   & \cite{ye2015} \\ \hline
\caption{Known convolution sums $W_{(\alpha, \beta)}(n)$} \label{introduction-table-1}
\end{longtable}

\begin{longtable}{|r|r|r|} \hline 
\textbf{$(a,b)$}  &  \textbf{Authors}   &  \textbf{References}  \\ \hline
(1,2)  & K.~S.~Williams   & \cite{williams2006} \\ \hline
(1,4)  & A.~Alaca \& S.~Alaca \& K.~S.~Williams   & \cite{alaca_alaca_williams2007} \\ \hline
(1,5)  & S.~Cooper \& D.~Ye   & \cite{cooper_ye2014} \\ \hline
(1,6)  & B.~Ramakrishman \& B.~Sahu   & \cite{ramakrishnan_sahu}  \\ \hline
(1,8)  & S.~Alaca \& Y.~Kesicio$\check{g}$lu   & \cite{alaca_kesicioglu2014} \\ \hline
\caption{Known representations of $n$ by the form \autoref{introduction-eq-1}
} 
\label{introduction-table-rep2}
\end{longtable}



\end{document}